\documentclass[english,12pt,oneside]{amsproc}
\usepackage[english]{babel}
\usepackage{a4wide}
\usepackage{amsthm}
\usepackage{graphics}
\usepackage{amsfonts, amssymb, amscd, amsmath}
\usepackage{latexsym}
\usepackage[matrix,arrow,curve]{xy}
\usepackage{mathabx}%,mathtools}
\usepackage{color}
\usepackage{pbox}
\usepackage{tikz}
\usetikzlibrary{matrix,decorations.pathreplacing,positioning}
\usepackage{hyperref}

 \DeclareMathOperator{\lk}{lk}

  \DeclareMathOperator{\ord}{Ord}
  
\DeclareMathOperator{\colim}{colim} \DeclareMathOperator{\hocolim}{hocolim}
\DeclareMathOperator{\rest}{rest}

\DeclareMathOperator{\birth}{birth}

\newcommand{\Zo}{\mathbb{Z}}
\newcommand{\Ro}{\mathbb{R}}

\newcommand{\Hr}{\widetilde{H}}

\newcommand{\ma}{\mu}
\newcommand{\na}{\nu}
\newcommand{\tG}{\tilde{G}}
\newcommand{\tD}{\tilde{D}}

\newcommand{\tb}{T_{\birth}}

\newcommand{\ST}[1]{\mbox{\upshape\small #1}}
\newcommand{\cat}{\ST{CAT}}
\newcommand{\Top}{\ST{TOP}}

\newcommand{\num}{w}

\newcounter{stmcounter}[section]
\newcounter{thcounter}

\numberwithin{equation}{section}

\theoremstyle{plain}
\newtheorem{cor}[stmcounter]{Corollary}

\newtheorem{thm}[thcounter]{Theorem}

\newtheorem{prop}[stmcounter]{Proposition}
\newtheorem{lem}[stmcounter]{Lemma}

\theoremstyle{definition}
\newtheorem{defin}[stmcounter]{Definition}

\theoremstyle{remark}
\newtheorem{ex}[stmcounter]{Example}
\newtheorem{rem}[stmcounter]{Remark}
\newtheorem{con}[stmcounter]{Construction}

\begin{document}

\title{Clique complexes of multigraphs, edge inflations, and tournaplexes}

\author{Anton Ayzenberg}
\address{Laboratory of algebraic topology and its applications, Faculty of computer science, Higher School of Economics}
\email{ayzenberga@gmail.com}

\author{Alexey Rukhovich}
\address{Skolkovo Institute of Science and Technology and Laboratory of algebraic topology and its applications, Faculty of computer science, Higher School of Economics}
\email{alex-ruhovich@mail.ru }

\date{\today}
\thanks{The article was prepared within the framework of the HSE University Basic Research Program}

\subjclass[2020]{Primary 55P10, 55U10, 55-08 ; Secondary 55N31, 55U05, 06A07, 55P40}

\keywords{Poset fiber theorems, multigraphs, digraphs, inflated complexes, Cohen--Macaulay complexes, persistent homology}

\begin{abstract}
In this paper we introduce and study the topology of clique complexes of multigraphs without loops. These clique complexes generalize tournaplexes, which were recently introduced by Govc, Levi, and Smith for the topological study of brain functional networks. We study a general construction of edge-inflated simplicial posets, which generalize clique complexes of multigraphs. The poset fiber theorem of Bjorner, Wachs, and Welker is applied to obtain the homotopy wedge decomposition of an edge-inflated simplicial poset. The homological corollary of this result allows to parallelize the homology computations for edge inflated complexes, in particular, clique complexes of multigraphs and tournaplexes. We provide functorial versions of some results to be used in computations of persistent homology. Finally, we introduce a general notion of simplex inflations and prove homotopy wedge decompositions for this class of spaces.
\end{abstract}

\maketitle

\section{Introduction}\label{secIntro}

The notion of a clique complex of a simple graph is well known in topological data analysis. The data arising from different fields of study often comes in the form of simple graphs, probably with additional structures, such as weights or distances. Clique complexes is a natural mean to transform graphs into higher dimensional objects which can be further studied with tools of algebraic topology: mainly simplicial and persistent homology.

However, sometimes the data comes in the form of more general discrete structures. Recently, there has been an arising interest in the analysis of these structures by topological methods. As a basic example, we mention the topological research at Blue Brain Project~\cite{BlBr1} which is concentrated around directed graphs $\Gamma$. There are two important topological constructions associated to $\Gamma$: the \emph{directed flag complex} $dF.\Gamma$ (the cell complex whose cells are the cliques of $\Gamma$ without directed cycles), and the \emph{tournaplex} $tF.\Gamma$ (the cell complex whose cells are all possible tournaments in $\Gamma$). Homology of the spaces $dF.\Gamma$ and $tF.\Gamma$ provide numerical features which can be used, for example, in the classification of directed networks' types.

In this paper we concentrate on studying the homotopy types of tournaplexes $tF.\Gamma$. One can notice that directions of arrows are actually irrelevant in the construction of a tournaplex: forgetting arrow directions in $\Gamma$ one gets a multigraph $\tilde{\Gamma}$ with multiplicity of each edge at most $2$ (multiple edges arise if we have both arrows $(i,j)$ and $(j,i)$ in $\Gamma$). We consider general multigraphs with multiple edges, but without loops, and introduce the notion of a clique complex of a multigraph. Even more generally, one can start with any simplicial complex or simplicial poset, and consider the operation which multiplies its edges in a natural way. We call this operation edge inflation in analogy with vertex inflations introduced in~\cite{BjWW}. The clique complex of a multigraph coincides with the edge-inflated clique complex of the underlying simple graph.

We show that edge-inflated simplicial posets admit homotopy wedge decompositions, whose wedge summands are the original simplicial complex and suspensions over links of its simplices, see Theorem~\ref{thmWedgeDecomposition}. This observation allows to parallelize computations of simplicial homology of edge-inflated posets: homology of each link can be computed separately. The general result applies to clique complexes of multigraphs and to tournaplexes.

A multigraph is called complete if each two vertices are connected by at least one edge. The homotopy wedge decomposition above relies on two main ingredients: the observation that the clique complex of a complete multigraph is a wedge of spheres, see Theorem~\ref{thmWedgeOfSpheres}, and the poset fiber theorem proved by Bj\"{o}rner, Wachs, and Welker~\cite{BjWW}. The first statement opens an interesting direction of research. A new invariant of a complete multigraph is introduced: the top Betti number of its clique complex. We give combinatorial criteria for this number to be equal to $0$ or $1$.

We prove a collection of related results. In order to simplify computations of persistent homology (as opposed to ordinary homology) one needs to prove the functoriality of the wedge decompositions. We prove the functorial version of the poset fiber theorem and use it to formulate and prove the functorial version of the wedge decomposition for edge-inflated simplicial posets. Next, we introduce the general notion of simplex inflations in a simplicial poset, and prove the analogue of the wedge decomposition. Finally, we show that Cohen--Macaulay property of a simplicial poset is preserved by inflations of simplices.

The paper is organized as follows. In Section~\ref{secSimpPosets} we give all necessary definitions. Two main results of the paper are formulated in Section~\ref{secTwoTheorems}. Theorem~\ref{thmWedgeOfSpheres} is proved in Section~\ref{secProofOfSphereWedge}, there we discuss the combinatorial properties of the new invariant of complete multigraphs. Theorem~\ref{thmWedgeDecomposition} is proved in Section~\ref{secProofOfWedgeDecomposition}, there we recall the poset fiber theorem. In Section~\ref{secFunctorialityAndPersistence}, we prove the functorial versions of the poset fiber theorem and of Theorem~\ref{thmWedgeDecomposition} and discuss how these results can be applied to parallelize persistent homology calculations. In Section~\ref{secGeneralInflations} we introduce the general notion of simplex-inflated simplicial poset and prove the wedge decomposition result for this construction. In the particular case of edge-inflated simplicial poset this proof gives an alternative (and probably simpler) viewpoint on the other results of the paper.

\section{Simplicial posets}\label{secSimpPosets}

\subsection{Simplicial complexes}

Let $K$ be a finite simplicial complex on a vertex set $V$. %This means that $K\subseteq 2^V$ is a collection of subsets of $V$ such that: (1) $\varnothing\in K$, (2) $I\in K$ and $J\subset I$ imply $J\in K$. The elements $I\in K$ are called simplices. The dimension $\dim I$ is equal to $|I|-1$ by definition, and $\dim K=\max_{I\in K}\dim I$. A simplicial complex of dimension $\leqslant 1$ is a simple graph.
%
%An abstract simplicial complex $K$ can be naturally turned into a topological space by considering its standard geometrical realization
%\begin{equation}\label{eqGeomRealiz}
%|K|=\bigcup_{I\in K}\triangle_I,\qquad \triangle_I=\conv\{e_i\mid i\in I\}.
%\end{equation}
%where $\{e_i\}$ is a basis of $\Ro^m$.
In the following sometimes we do not make a distinction between simplicial complexes and their geometrical realizations. In particular, the notation $K\simeq L$ means that geometrical realizations of the complexes $K$ and $L$ are homotopy equivalent.

Recall that for any simple graph $G$ there is a simplicial complex $F.G$ called the clique complex of $G$ whose simplices are the cliques in $G$.

\subsection{Simplicial posets}

Let $S$ be a finite partially ordered set (shortly, poset). %The morphism of posets $f\colon S\to T$ is a monotonic map: if $s_1\leqslant s_2$ in $S$, then $f(s_1)\leqslant f(s_2)$ in $T$.
For $s\in S$ consider the posets $S_{\geqslant s}=\{s'\in S\mid s'\geqslant s\}$, $S_{\leqslant s}=\{s'\in S\mid s'\leqslant s\}$ and similarly for $S_{>s}$ and $S_{<s}$.

\begin{defin}\label{defOrderComplex}
Let $\ord(S)$ denote the simplical complex on the set $S$, whose simplices are the linearly ordered subposets (chains) in $S$. $\ord(S)$ is called \emph{the order complex} of the poset $S$. The geometrical realization $|S|=|\ord S|$ is called the \emph{geometrical realization} of~$S$.
\end{defin}

\begin{defin}\label{definSimpPoset}
A poset $S$ is called \emph{simplicial} if, for any element $I\in S$, the lower order ideal $S_{\leqslant I}$ is isomorphic to the poset of faces of a $k$-dimensional simplex. The number $k$ is called the dimension of $I$ and denoted $\dim I$. Elements of $S$ are called simplices, simplices of dimension $0$ are called vertices, simplices of dimension $1$ are called edges.
\end{defin}

Any simplicial complex is a simplicial poset. If $X$ is either a simplicial poset or a simplex of a simplicial poset, we denote the vertex set of $X$ by $V(X)$ and the edge set by $E(X)$. If $I\in S$ is a simplex, then $|V(I)|=\dim I+1$. The edge set $E(I)$ of a simplex $I$ can be identified with the set of pairs ${V(I)\choose 2}$.

\begin{defin}\label{definLinkInPoset}
Let $I$ be a simplex of a simplicial complex $S$. The poset $\lk_SI=S_{>I}$ is called the \emph{link} of $I$ in $S$. Link is again a simplicial poset.
\end{defin}

This definition is compatible with the standard definition of link in a simplicial complex.

\subsection{Multigraphs}

Sometimes the term ``multigraph'' refers to the analogues of graphs having loops and multiple edges. We want to allow multiple edges, but not the loops, and call such objects multigraphs.

\begin{defin}
A \emph{multigraph} is a 1-dimensional simplicial poset.
\end{defin}

Let $\tilde{G}$ be a multigraph with the vertex set $V$. Let $G$ be the underlying graph of $\tilde{G}$, i.e. the simple graph on the vertex set $V$ obtained by forgetting multiplicities of edges. A multigraph $\tilde{G}$ is encoded by a pair $(G,\ma)$, where $G$ is a simple graph, and the function $\ma\colon E(G)\to \Zo_+$ assigns positive multiplicities to edges of $G$.

The following statement is quite simple and straightforward, see Fig.~\ref{figMultigraph}, but it gives a flavour of our main result.

\begin{prop}
Let $G$ (hence $\tilde{G}$) be connected. Then
\[
|\tilde{G}|\simeq |G|\vee \bigvee_{e\in E(G)}(S^1)^{\vee(\ma(e)-1)}.
\]
Here $X^{\vee k}$ denotes the wedge of $k$ copies of a space $X$.
\end{prop}

\begin{figure}[h]
\begin{center}
\includegraphics[scale=0.2]{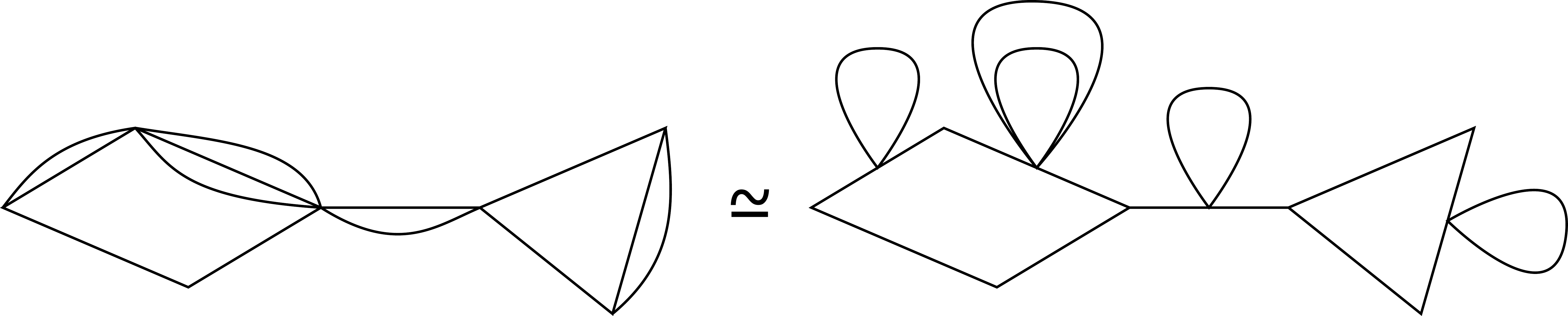}
\end{center}
\caption{Topology of a multigraph}\label{figMultigraph}
\end{figure}

In other words, the homotopy type of a multigraph depends on the homotopy type of the underlying graph and on the combinatorics of multiple edges in a controllable way. Our main result, Theorem~\ref{thmWedgeDecomposition} transfers this idea to higher dimensions.

\subsection{Edge-inflated simplicial posets}

Let us fix an edge multiplicity function \[\ma\colon E(S)\to \Zo_+.\]

\begin{defin}
Let $S$ be a simplicial poset and $\ma\colon E(S)\to \Zo_+$ be an edge multiplicity function. Consider the poset $S_\ma$ defined as follows. The elements of $S_\ma$ are the pairs $(I,c_I)$, where $I\in S$ is a simplex of $S$ and $c_I\colon E(I)\to\Zo_+$ is a ``sheet-counting function'', satisfying the condition: for any edge $e\leqslant I$, $c_I(e)\in\{1,\ldots,\ma(e)\}$. The partial order on $S_\ma$ is defined as follows: $(I,c_I)<(J,c_J)$ if $I<J$ and $c_I$ equals the restriction of $c_J$ to the subset $E(I)\subset E(J)$. We call $S_\ma$ the \emph{edge inflation} of $S$.
\end{defin}

The definition implies that sheet-counting function of an element in $S_\ma$ determines sheet-counting functions of lower elements uniquely. Hence $S_\ma$ is again a simplicial poset.

Stating things in simple words, we make copies of each edge of $S$ according to the multiplicity function, and each collection of copies of edges spans its own copy of each higher-dimensional simplex.

\begin{rem}
If a multigraph $\tilde{G}$ is encoded by a pair $(G,\ma)$, then $\tilde{G}=G_\ma$.
\end{rem}

\begin{defin}
Let $\tilde{G}=G_\ma$ be a multigraph. A \emph{clique} of a multigraph $\tilde{G}$ is a sub-multigraph of $\tilde{G}$ isomorphic to a complete simple graph. The \emph{clique complex} $F.\tilde{G}$ of a multigraph $\tilde{G}$ is the poset of all cliques ordered by inclusion. The poset $F.\tilde{G}$ is simplicial.
\end{defin}

The next lemma easily follows from the definitions.

\begin{lem}
Let $K$ be the clique complex of a simple graph $G$, $K=F.G$. Then for any edge multiplicity function $\mu$ the poset $K_\ma$ is isomorphic to the clique complex of the multigraph $G_\ma$.
\end{lem}

Recall that a \emph{digraph} is a collection of vertices together with a set of ordered pairs of vertices $(v_1,v_2)$ such that $v_1\neq v_2$.

\begin{defin}
A \emph{tournament} is a digraph on a vertex set $V$, in which, for any two vertices $v\neq w$, there exists either an arrow $(v,w)$ or an arrow $(w,v)$ but not both. If $\Gamma$ is a digraph, then a directed subgraph of $\Gamma$ isomorphic to a tournament is a called a \emph{tournament in $\Gamma$}.
\end{defin}

\begin{defin}
Let $\Gamma$ be a digraph. The \emph{flag tournaplex} $tF.\Gamma$ is a partially ordered set of tournaments of $\Gamma$. This poset is simplicial.
\end{defin}

\begin{rem}
It can be seen that directions of arrows are irrelevant in the definition of the tournaplex. Forgetting the directions, we can turn any digraph $\Gamma$ into a multigraph $\tilde{\Gamma}$, where the multiplicity $\ma(e)$ of each edge is at most $2$. Then the tournaplex $tF.\Gamma$ coincides with the clique complex $F.\tilde{\Gamma}$ simply by definition.
\end{rem}

\section{Wedge decompositions}\label{secTwoTheorems}

Let us state the main results of this paper related to edge inflated simplicial posets.

\begin{thm}\label{thmWedgeOfSpheres}
Let $\Delta_V$ be the simplex on a vertex set $V$, $|V|=n$, considered as a simplicial complex. Then, for any multiplicity function $\ma\colon E(\Delta_V)={V\choose 2}\to \Zo_+$, the edge inflated simplicial poset $(\Delta_V)_\ma$ is homotopy equivalent to a wedge of spheres of dimension $n-1$.
\end{thm}

The number $n(\ma,V)$ of spheres in this wedge can be easily computed via Euler characteristic as follows. Each simplex $J\subseteq V$ generates $\prod_{e\in {J\choose 2}} \ma(e)$ copies in $(\Delta_V)_\ma$. Therefore the number $f_k=f_k((\Delta_V)_\ma)$ of $k$-dimensional simplices is given by
\[
f_k=\sum_{J\subseteq V, |J|=k+1} \prod_{e\in {J\choose 2}} \ma(e)\quad \mbox{ for } k>0
\]
and $f_0((\Delta_V)_\ma)=n$. Hence
\begin{multline}\label{eqNumberInWedge}
n(\ma,V)=f_{n-1}-f_{n-2}+\cdots+(-1)^{n-2}f_2+(-1)^{n-1}n+(-1)^n =\\ =
\sum_{J\subseteq V, |J|\geqslant 2}(-1)^{n-|J|}\prod_{e\in{J\choose 2}}\ma(e) +(-1)^{n-1}(n-1).
\end{multline}

\begin{thm}\label{thmWedgeDecomposition}
Let $S$ be a connected simplicial poset. Then, for any edge-multiplicity function $\ma\colon E(S)\to\Zo_+$, there is a homotopy equivalence
\begin{equation}\label{eqWedgeInflated}
|S_\ma|\simeq |S|\vee \bigvee_{I\in S, \dim I\geqslant 1}(\Sigma^{\dim I+1}|\lk_SI|)^{\vee n(\ma_I,I)}.
\end{equation}
Here $n(\ma_I,I)$ denotes the number given by~\eqref{eqNumberInWedge} for the function $\ma$ restricted to the edge set of $I$, and $X^{\vee k}$ denotes the wedge of $k$ copies of $X$.
\end{thm}

The homological statement easily follows.

\begin{cor}\label{corComputational}
For any edge multiplicity function $\mu$ on $S$ there holds
\[
\Hr_*(S_\ma)\cong \Hr_*(S)\oplus \bigoplus_{I\in S, \dim I\geqslant 1}n(\ma,I)\Hr_{*-\dim I-1}(\lk_SI).
\]
\end{cor}

In the above two statements we use the following conventions which are standard in combinatorial topology. For the empty poset $\varnothing$ there holds $\Sigma\varnothing=S^0$ (and consequently $\Sigma^k\varnothing\cong S^{k-1}$). It is assumed that $\Hr_{-1}(\varnothing;R)\cong R$ while $\Hr_j(\varnothing;R)=0$ for $j\neq -1$. The suspension isomorphism in homology holds for empty spaces with this convention. This observation becomes essential, since for each maximal simplex $I\in S$ we have $\lk_SI=\varnothing$ appearing in the righthand side of the decompositions.

The formula of Corollary~\ref{corComputational} allows to perform parallel computations of simplicial homology of edge-inflated simplicial posets, in particular, the homology of clique complexes of multigraphs and flag tournaplexes of digraphs.
%
%\begin{rem}
%About empty set and its reduced homology
%\end{rem}

\section{Proof of Theorem~\ref{thmWedgeOfSpheres} and related results}\label{secProofOfSphereWedge}

In the following, by a topological space we mean a (finite) CW-complex. All maps are assumed cellular and all pairs --- CW-pairs. Since we are dealing with geometrical realizations of posets, which are naturally CW-complexes, these assumptions do not restrict the generality of arguments in our context. Recall that a space $X$ is called $j$-connected if $\pi_i(X)=1$ for all $i\leqslant j$.

\begin{lem}\label{lemStackSpheres}
Let $X_1,X_2$ be subcomplexes of $Y$ such that $Y=X_1\cup X_2$ and the following conditions hold: (1) $X_1\simeq (S^k)^{\vee r_1}$,  $X_2\simeq(S^k)^{\vee r_2}$; (2) $A=X_1\cap X_2\simeq (S^{k-1})^{\vee r_3}$. Then $Y$ is homotopy equivalent to the wedge of $r_1+r_2+r_3$ many $k$-dimensional spheres.
\end{lem}

\begin{proof}
By definition, $Y=X_1\cup X_2$ is a colimit of the diagram $X_1\hookleftarrow A \hookrightarrow X_2$. This diagram is cofibrant, hence
\[
Y=\colim(X_1\hookleftarrow A \hookrightarrow X_2)\simeq \hocolim(X_1\hookleftarrow A \hookrightarrow X_2).
\]
The homotopy colimit is a union of the mapping cylinders of inclusions $\imath_1\colon A\hookrightarrow X_1$ and $\imath_2\colon A\hookrightarrow X_2$ along $A$. Since $X_1$ is $(k-1)$-connected and $A$ is a wedge of $(k-1)$-spheres, the image $\imath_1(A)$ can be contracted to a point inside $X_1$, and similarly for $X_2$. Therefore,
\[
\hocolim(X_1\hookleftarrow A \hookrightarrow X_2)\simeq X_1\vee \Sigma A\vee X_2.
\]
The statement of lemma now easily follows.
\end{proof}

We prove Theorem~\ref{thmWedgeOfSpheres} by induction on $n=|V|$. The cases $n=0$ and $n=1$ are straightforward.

\begin{lem}\label{lemOverDiamond}
Assume that the statement of Theorem~\ref{thmWedgeOfSpheres} holds for all $n\leqslant k$. Let $L=\Delta_{[k-1]}\ast S^0$ be the simplicial complex on $k+1$ vertices obtained by stacking two $(k-1)$-simplices $\Delta_1$ and $\Delta_2$ along one facet $\Delta'=\Delta_{[k-1]}$. Then, for any multiplicity function $\ma\colon E(L)\to\Zo_+$, the edge-inflated complex $L_\ma$ has homotopy type of a wedge of $k$-dimensional spheres.
\end{lem}

\begin{proof}
Apply Lemma~\ref{lemStackSpheres} to $L_\ma=(\Delta_1)_\ma\cup_{\Delta'_\ma}(\Delta_2)_\ma$.
\end{proof}

\begin{proof}[Proof of Theorem~\ref{thmWedgeOfSpheres}]
Assume that the statement of Theorem~\ref{thmWedgeOfSpheres} holds for all $n\leqslant k$. Let us prove it for $n=k+1$. Consider a multiplicity function $\ma\colon {V\choose 2}=E(\Delta_V)\to \Zo_+$, where $|V|=k+1$. We will use inner induction on the value $s(\ma)=\sum_{e\in E(\Delta_V)}\ma(e)$. There is an alternative:

(1) $\ma(e)=1$ for each $e\in {V\choose 2}$. In this case there is nothing to prove, since $(\Delta_V)_\ma=\Delta_V$ is contractible. This establishes the base of induction, $s(\ma)={k+1\choose 2}$.

(2) There exists an edge $e\in E(\Delta_V)$ such that $\ma(e)>1$. Without loss of generality assume $e=\{1,2\}$. Let $M=\{1,2,\ldots, \ma(e)\}$ denote the set of copies of $e$ in $(\Delta_V)_\ma$, so that $|M|=\ma(e)>1$. Let us represent $M$ as a disjoint union $M_1\sqcup M_2$, such that $0<|M_1|,|M_2|<\ma(e)$. Consider the following subcomplexes in $(\Delta_V)_\ma$:
\[
K_1=\{(I,c_I)\in (\Delta_V)_\ma\mid c_I(e)\in M_1\mbox{ if }e\in I\},\qquad K_2=\{(I,c_I)\in (\Delta_V)_\ma\mid c_I(e)\in M_2\mbox{ if }e\in I\};
\]
We see that $(\Delta_V)_\ma=K_1\cup K_2$. From the construction, it follows that $K_1\cong (\Delta_V)_{\ma_1}$ and $K_2\cong (\Delta_V)_{\ma_2}$, where $\ma_1, \ma_2$ coincide with $\ma$ on all edges except $e$, while $\ma_1(e)=|M_1|$, $\ma_2(e)=|M_2|$. Therefore, by inner induction, both complexes $K_1$ and $K_2$ are homotopy equivalent to wedges of $k$-dimensional spheres.

Since $c_I(e)$ cannot lie simultaneously in $M_1$ and $M_2$, the intersection $K_1\cap K_2$ has the form
\[
K_1\cap K_2=\{(I,c_I)\in (\Delta_V)_\ma\mid e\notin I\}.
\]
This complex coincides with $L_\ma$ where $L = \Delta_1\cup\Delta_2$, $\Delta_1 = \Delta_{V\setminus\{1\}}$, and $\Delta_2 = \Delta_{V\setminus\{2\}}$ (recall that $e=\{1,2\}$). Therefore, Lemma~\ref{lemOverDiamond} applies: the intersection $K_1\cap K_2=L_\ma$ has homotopy type of a wedge of $(k-1)$-dimensional spheres by induction hypothesis.

Applying Lemma~\ref{lemStackSpheres} to the complexes $K_1,K_2$, we prove that $K_1\cup K_2=(\Delta_V)_\ma$ is homotopy equivalent to a wedge of $k$-spheres. This finishes the induction step in the proof of Theorem~\ref{thmWedgeOfSpheres}.
\end{proof}

\begin{rem}
One can extract the exact number of spheres in wedges from the inductive argument used in the proof above. Let $\num(X)$ denote the number of spheres in $X\simeq\bigvee S^k$. From the proof of Lemma~\ref{lemStackSpheres} it follows that
\begin{equation}\label{eqNumberInUnion}
\num(Y) = \num(X_1) + \num(X_2) + \num(A),
\end{equation}
in the notation of the lemma.
\end{rem}

\begin{lem}\label{lemInductiveCalculation}
Let $\ma\colon {V\choose 2}\to\Zo_+$ be a multiplicity function on the edge set of the simplex $\Delta_V$. Let $e=\{v_1,v_2\}$ be an edge, $r=\ma(e)$, and let $\ma_1\colon {V\choose 2}\to\Zo_+$ be the multiplicity function which coincides with $\ma$ on all edges except $e$, while $\ma_1(e)=1$. Then
\[
n(\ma,V)=rn(\ma_1,V)+(r-1)[n(\ma,V\setminus\{v_1\})+n(\ma,V\setminus\{v_2\})+n(\ma,V\setminus\{v_1,v_2\})]
\]
\end{lem}

\begin{proof}
The same argument as in the proof of Theorem~\ref{thmWedgeOfSpheres} shows that
\begin{equation}\label{eqTechButUseful}
n(\ma,V)=\num((\Delta_V)_\ma)=\num((\Delta_V)_{\ma_{r-1}})+\num((\Delta_V)_{\ma_1})+\num(L_\ma)
\end{equation}
where $\ma_{r-1}$ coincides with $\ma$ outside $e$ and $\ma_{r-1}(e)=r-1$; and $L$ is obtained from $\Delta_V$ by deleting the edge $e$. Therefore induction on $r$ implies
\[
n(\ma,V)=rn(\ma_1,V)+(r-1)\num(L_\ma).
\]
Since $L$ is the union of the simplices $\Delta_{V\setminus\{v_1\}}$ and $\Delta_{V\setminus\{v_2\}}$ along their facet $\Delta_{V\setminus\{v_1,v_2\}}$, formula~\eqref{eqNumberInUnion} implies
\[
\num(L_\ma)=n(\ma,V\setminus\{v_1\})+n(\ma,V\setminus\{v_2\})+n(\ma,V\setminus\{v_1,v_2\}),
\]
which completes the proof.
\end{proof}

\begin{ex}
We demonstrate the calculation of the number of spheres in the wedges (as well as the proof of Theorem~\ref{thmWedgeOfSpheres}) on a simple example. Figure~\ref{figArithmetics} shows the number of spheres of the multigraph clique complex: the notation $w(\tilde{G})$ means $\beta_{\dim F.\tilde{G}}(F.\tilde{G})$. It can be seen that the result of this computation coincides with the result of formula~\eqref{eqNumberInWedge}.
\end{ex}

\begin{figure}[h]
\begin{center}
\includegraphics[scale=0.25]{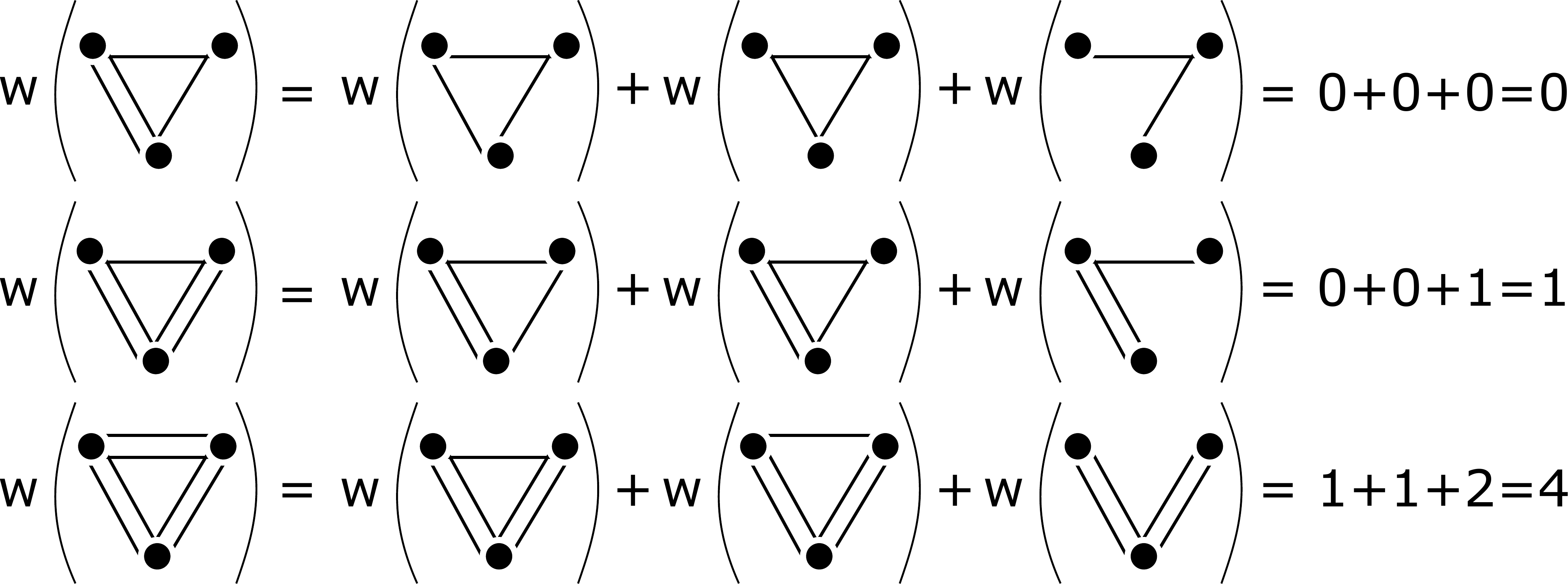}
\end{center}
\caption{Computation of the top Betti number of the multigraph flag complexes}\label{figArithmetics}
\end{figure}

Let us prove several more statements about complexes of the form $(\Delta_V)_\ma$, i.e. clique complexes of complete multigraphs.

\begin{prop}\label{propMonot}
Let $\ma_1,\ma_2\colon{V\choose 2}\to\Zo_+$ be two multiplicity functions such that $\ma_1\preceq\ma_2$ (i.e. $\ma_1(e)\leqslant \ma_2(e)$ for any $e\in{V \choose 2}$). Then the natural inclusion $(\Delta_V)_{\ma_1}\hookrightarrow (\Delta_V)_{\ma_2}$ is homotopy equivalent to an inclusion of subwedge (in other words, the inclusion induces injective map on the top-degree homology modules).
\end{prop}

The proof by induction follows the same lines as the proof of Theorem~\ref{thmWedgeOfSpheres}. There is a natural analogue of Lemma~\ref{lemStackSpheres}, proved the same way.

\begin{lem}\label{lemStackSpheresMorphism}
Let $X_1,X_2$ be subcomplexes of $Y=X_1\cup X_2$, $A=X_1\cap X_2$, $\tilde{X}_1,\tilde{X}_2$ be subcomplexes of $\tilde{Y}=\tilde{X}_1\cup \tilde{X}_2$, $\tilde{A}=\tilde{X}_1\cap \tilde{X}_2$, and there are maps $A\to \tilde{A}$, $X_1\to \tilde{X}_1$, $X_2\to \tilde{X}_2$, $Y\to \tilde{Y}$, commuting with all inclusions. Assume that $X_1\to \tilde{X}_1$ and $X_2\to \tilde{X}_2$ are homotopy equivalent to a inclusions of subwedges of $k$-dimensional spheres into wedges of $k$-dimensional spheres, and $A\to\tilde{A}$ is homotopy equivalent to an inclusion of a subwedge of $(k-1)$-dimensional spheres into a wedge of $(k-1)$-dimensional spheres. Then $Y\to\tilde{Y}$ is also homotopy equivalent to an inclusion of subwedge of $k$-spheres.
\end{lem}

This allows to perform induction steps as in the proof of Theorem~\ref{thmWedgeOfSpheres}.

\begin{cor}\label{corMonot}
Assume $V$ is fixed. Then the number $n(\ma,V)$ of wedge summands is nonstrictly monotonic with respect to $\ma$.
\end{cor}

This is difficult to see directly from formula~\eqref{eqNumberInWedge}, however this can be deduced from formula~\eqref{eqTechButUseful}.

Let us classify multiplicity functions which give rise to contractible spaces, and to single spheres. First of all, we have a useful topological observation.

\begin{prop}\label{propComponents}
Assume that the vertex set $V$ is decomposed into two nonempty parts $V=V_1\sqcup V_2$, such that $\ma(e)=1$ for any edge between different parts. Then $(\Delta_V)_\ma\cong (\Delta_{V_1})_{\ma_1}\ast (\Delta_{V_2})_{\ma_2}$, where $\ma_1,\ma_2$ are the restrictions of $\ma$ to ${V_1\choose 2}$ and ${V_2 \choose 2}$ respectively.
\end{prop}

\begin{proof}
This follows easily from the definition of the join operation: whenever we have two simplicial posets $K_1$ and $K_2$ their join is formed by simplices of the form $J_1\ast J_2$ where $J_1\in K_1$ and $J_2\in K_2$. If $J_1$ and $J_2$ are simplices in $(\Delta_{V_1})_{\ma_1}$ and $(\Delta_{V_2})_{\ma_2}$ their join $J_1\ast J_2$ is by definition a simplex in $(\Delta_V)_\ma$. Conversely, any simplex of $(\Delta_V)_\ma$ decomposes uniquely in the join of simplices in $(\Delta_{V_1})_{\ma_1}$ and $(\Delta_{V_2})_{\ma_2}$ (the uniqueness follows from the fact that all edges between different parts have multiplicity $1$).
\end{proof}

\begin{cor}\label{corIfDisjVertexThenContractible}
If there is a vertex $v\in V$ such that all its adjacent edges have multiplicity 1, then $(\Delta_V)_\ma$ is contractible.
\end{cor}

\begin{con}
Let $\ma\colon{V\choose 2}\to\Zo_+$ be a multiplicity function. Consider the graph $G(\ma)$ on the vertex set $V$ whose edges are the pairs $e\in {V\choose 2}$ such that $\ma(e)>1$. Proposition~\ref{propComponents} states that $(\Delta_V)_\ma$ is the join of complexes, corresponding to the connected components of $G(\ma)$. If there is a disjoint vertex in $G(\ma)$, then $(\Delta_V)_\ma$ is contractible. In Proposition~\ref{propContractibleCriterion} below we prove that this condition is also necessary for the contractibility of $(\Delta_V)_\ma$.
\end{con}

\begin{defin}
A multiplicity function $\ma$ is called a \emph{starry sky function}, if $\ma(e)\leqslant 2$ for any $e\in{V\choose 2}$ and all connected components of $G(\ma)$ are isomorphic to star graphs with at least 2 vertices.
\end{defin}

\begin{figure}[h]
\begin{center}
\includegraphics[scale=0.35]{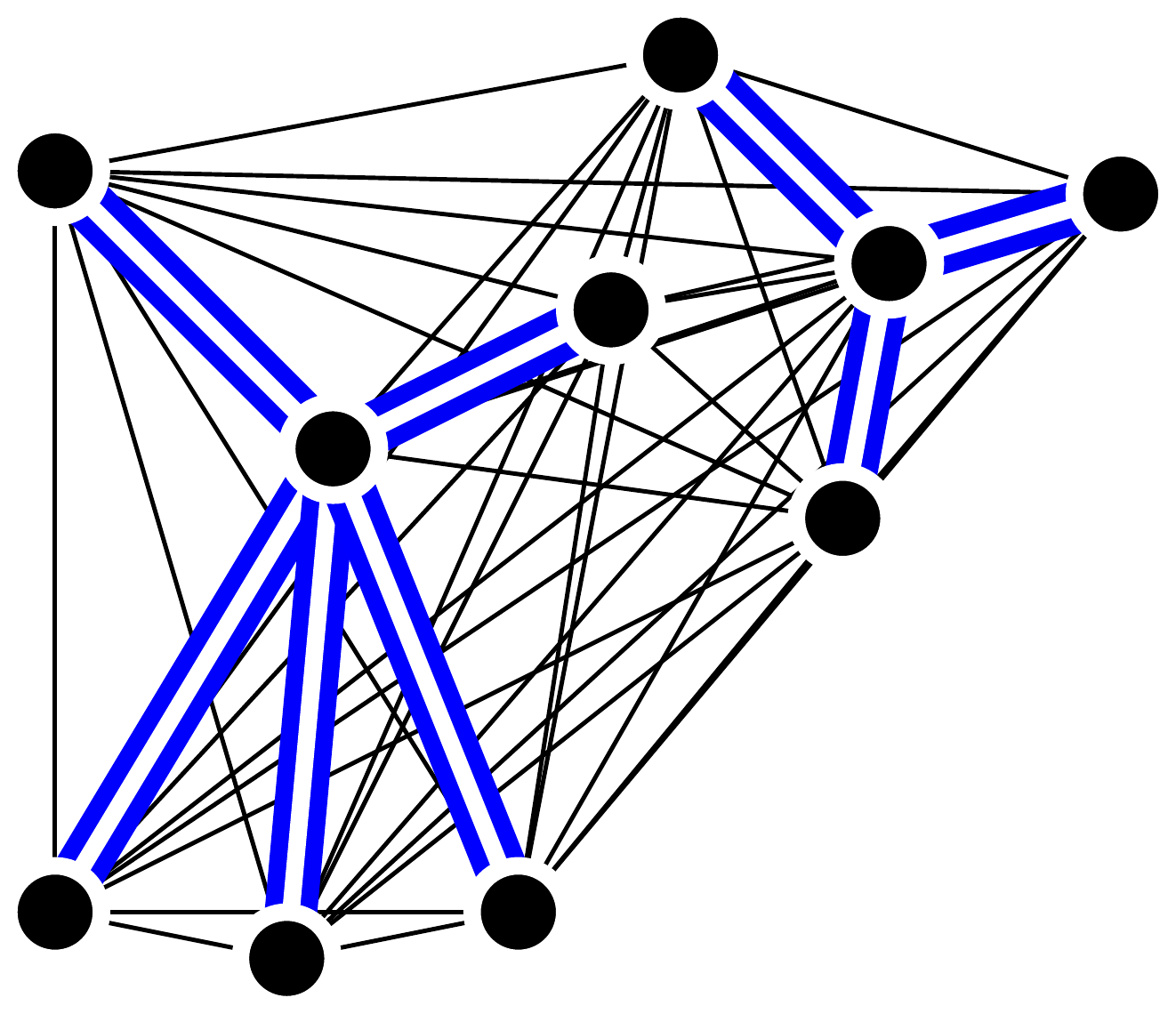}
\end{center}
\caption{Starry sky multiplicity function with two star components}\label{figStarry}
\end{figure}

\begin{lem}\label{lemStarry}
If $\ma$ is a starry sky function, then $(\Delta_V)_\ma$ is homotopy equivalent to a single sphere.
\end{lem}

\begin{proof}
According to Proposition~\ref{propComponents}, it is sufficient to prove the statement in the case when $G(\ma)$ has a single connected component, isomorphic to a star graph. Indeed, the join of two spheres is again a sphere. The multigraph corresponding to $\ma$ is as shown on Fig.~\ref{figStarry}. In this case we use induction on the number of vertices of the star. The base, $|V|=2$ corresponds to the multigraph, which has 2 edges between two vertices, its geometrical realization is the circle $S^1$. The induction step follows from Lemma~\ref{lemInductiveCalculation}: the calculation is shown on Fig.~\ref{figStarArithmetics}.
\end{proof}

\begin{figure}[h]
\begin{center}
\includegraphics[scale=0.2]{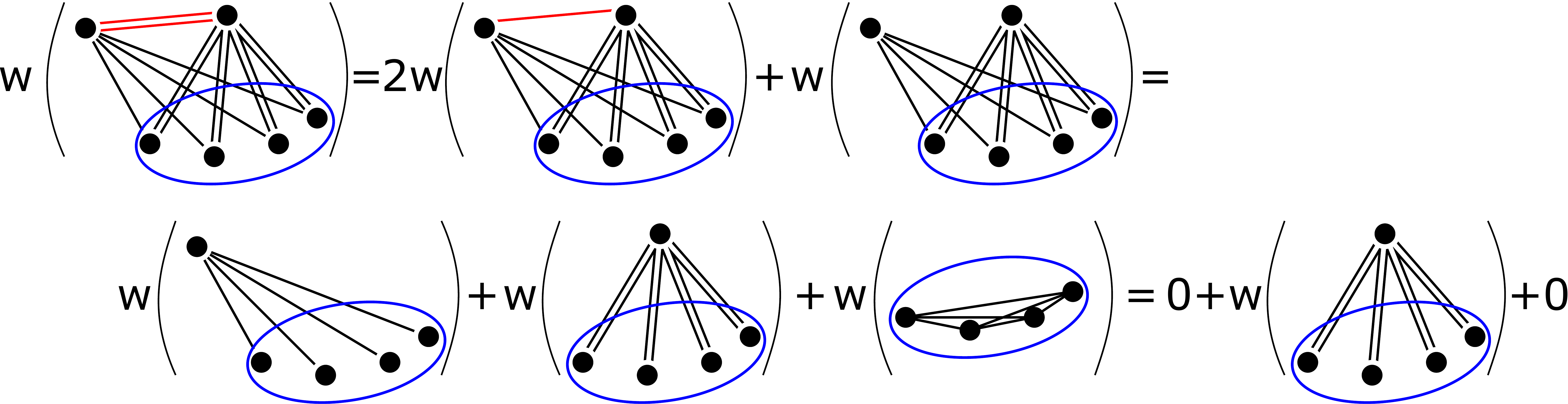}
\end{center}
\caption{Calculation with the starry sky multiplicity function. An ellipse is a complete subgraph with single edges.}\label{figStarArithmetics}
\end{figure}

\begin{lem}\label{lemConGraphsStronger}
Let $\ma\colon {V\choose 2}\to\Zo_+$ be an edge multiplicity function such that $G(\ma)$ is connected. Let $V'\subset V$ be a subset such that $G(\ma)|_{V'}$ is connected as well. Then $n(\ma,V)\geqslant n(\ma,V')$.
\end{lem}

\begin{proof}
Without loss of generality we may assume $|V|-|V'|=1$, that is $V\setminus V'$ is a single vertex $v$. The connectedness condition implies there is an edge $e$ of $G(\ma)$ from $v$ to some vertex in $V'$. Since $\ma(e)\geqslant 2$, the required inequality follows from Lemma~\ref{lemInductiveCalculation} applied to the edge $e$.
\end{proof}

\begin{rem}
Note that Lemma~\ref{lemConGraphsStronger} does not follow from Proposition~\ref{propMonot}. In proposition we deal with edge multiplicity functions on the same set $V$, while lemma allows to compare the invariants $n(\ma,V)$ for sets $V$ of different cardinalities.
\end{rem}

\begin{prop}\label{propContractibleCriterion}
$(\Delta_V)_\ma$ is contractible if and only if $G(\ma)$ has a disjoint vertex.
\end{prop}

\begin{proof}
One direction is already proved in Corollary~\ref{corIfDisjVertexThenContractible}. Now assume that all vertices of $G(\ma)$ have degree at least 1. Let $\{F_i\}$ be connected components of $G(\ma)$ on the vertex sets $V_i$. Each $F_i$ has at least 2 vertices. Hence $F_i$ contains a connected subgraph $F_i'$ which consists of a single edge $e_i=\{v_i,w_i\}$. Since $\ma(e_i)\geqslant 2$, we have $n(\ma,\{v_i,w_i\})=\ma(e)-1\geqslant 1$. Applying Lemma~\ref{lemConGraphsStronger} to each connected component $F_i$ we see that \[
n(\ma,V)=\prod_{i}n(\ma,V_i)\geqslant \prod_in(\ma,\{v_i,w_i\})>0.
\]
This proves the statement.
\end{proof}

\begin{prop}\label{propOneSphere}
$(\Delta_V)_\ma$ is homotopy equivalent to a single sphere if and only if $\ma$ is a starry sky function.
\end{prop}

\begin{proof}
One direction is already proved in Lemma~\ref{lemStarry}. To prove the other direction it suffices to prove the following statements.
\begin{enumerate}
  \item If there is an edge $e\in{V\choose 2}$ of multiplicity $\ma(e)\geqslant 3$ in a connected graph $G(\ma)$, then $n(\ma,V)>1$.
  \item If a connected graph $G(\ma)$ contains a subgraph isomorphic to a path with $4$ vertices (the 4-path), then $n(\ma,V)>1$.
  \item If a connected graph $G(\ma)$ contains a 3-cycle, then $n(\ma,V)>1$.
\end{enumerate}
If we want $n(\ma,V)$ to be equal to $1$, then (1) implies that all multiplicities should be at most $2$ and each connected component of $G(\ma)$ should be 4-path-free by (2), and 3-cycle-free by (3). This means that each connected component is a star.

The proof of statements (1)-(3) is based on monotonicity properties (Proposition~\ref{propMonot} and Lemma~\ref{lemConGraphsStronger}) and direct calculations shown on Fig.~\ref{figExceptions}.
\end{proof}

\begin{figure}[h]
\begin{center}
\includegraphics[scale=0.2]{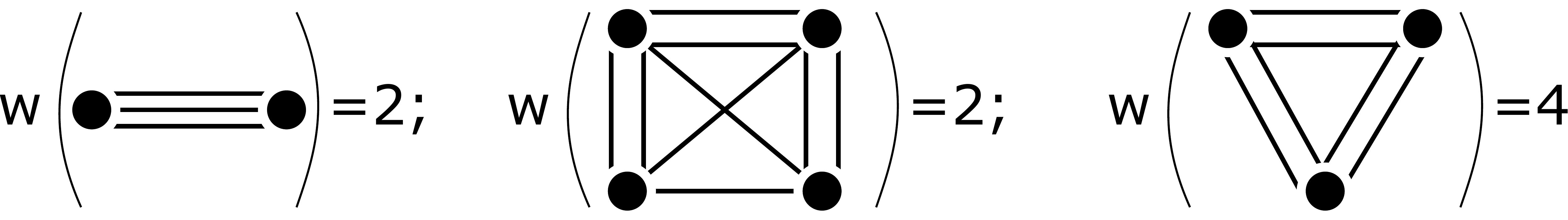}
\end{center}
\caption{Exceptional multigraphs from the proof of Proposition~\ref{propOneSphere}}\label{figExceptions}
\end{figure}

\section{Proof of Theorem~\ref{thmWedgeDecomposition} and related results}\label{secProofOfWedgeDecomposition}

Let us recall the general result of Bj\"{o}rner, Wachs, and Welker, the poset fiber theorem.

\begin{thm}[{\cite[Thm.1.1]{BjWW}}]\label{thmFiberThm}
Let $f\colon P\to Q$ be a morphism of posets such that for all $q\in Q$ the fiber $|f^{-1}(Q_{\leqslant q})|$ is $\dim |f^{-1}(Q_{<q})|$-connected. If $|Q|$ is connected, then there is a homotopy equivalence
\begin{equation}\label{eqWedgeGeneral}
|P|\simeq |Q|\vee\bigvee_{q\in Q} |f^{-1}(Q_{\leqslant q})|\ast |Q_{>q}|.
\end{equation}
\end{thm}

If $Q$ is not connected, this result applies separately to each connected component of $Q$.

We use this result to prove Theorem~\ref{thmWedgeDecomposition}. The idea is essentially the same as the one used in the study of inflated (or vertex-inflated) complexes~\cite[Sect.6]{BjWW}, see Definition~\ref{definVertexInflation} below.

\begin{proof}[Proof of Theorem~\ref{thmWedgeDecomposition}]
Consider the map $f\colon S_\ma\to S$, which maps an element $(I,c_I)$ to $I$ (collapses all copies of $I$ back to a single simplex). The preimage of any vertex is a single point. Now let $I\in S$ be a simplex of dimension at least $1$ with the vertex set $V(I)$, and let $\ma_I$ be the restriction of $\ma$ to the edges of $I$. Then the fiber $f^{-1}(S_{\leqslant I})$ coincides with the poset $(\Delta_{V(I)})_{\ma_I}$. This poset is homotopy equivalent to a wedge of $\dim I$-dimensional spheres by Theorem~\ref{thmWedgeOfSpheres}, hence it is $(\dim I-1)$-connected. Therefore, we are in position to apply Theorem~\ref{thmFiberThm}. The wedge decomposition of this theorem asserts
\[
|S_\ma|\simeq |S|\vee\bigvee_{I\in S}|(\Delta_{V(I)})_{\ma_I}|\ast |S_{>I}|.
\]
It remains to observe the following.
\begin{enumerate}
  \item If $I$ is a vertex, then $|(\Delta_{V(I)})_{\ma_I}|$ is also a single vertex, hence the corresponding summands of the wedge decomposition disappear.
  \item $S_{>I}=\lk_SI$ by the definition of a link in a simplicial poset.
  \item $|(\Delta_{V(I)})_{\ma_I}|\simeq (S^{\dim I})^{\vee n(\ma_I,I)}$ by Theorem~\ref{thmWedgeOfSpheres}. Therefore
  \[
  |(\Delta_{V(I)})_{\ma_I}|\ast |S_{>I}|\simeq (\Sigma^{\dim I+1}\lk_SI)^{\vee n(\ma_I,I)}.
  \]
\end{enumerate}
This finishes the proof of Theorem~\ref{thmWedgeDecomposition}.
\end{proof}

Our next goal is to characterize links of simplices in an edge-inflated simplicial poset. For this goal we recall the notion of vertex-inflated simplicial complexes (or posets).

Let $S$ be a simplicial poset on the vertex set $V(S)$. A function $\na\colon V(S)\to \Zo_+$ is called a \emph{vertex multiplicity function}.

\begin{defin}\label{definVertexInflation}
Consider the poset $S_{\na}$ whose elements have the form $(I,d_I)$, where $I\in S$ and $d_I\colon V(I)\to \Zo_+$ is a sheet counting function such that $d_I(v)\in\{1,\ldots,\na(v)\}$ for each $v\in V(S)$. The partial order is introduced on $S_{\na}$ in a natural way: $(I,d_I)\leqslant (J,d_J)$ if $I\leqslant J$ and $d_I$ equals the restriction of $d_J$ to $V(I)\subseteq V(J)$. The poset $S_\na$ is a simplicial poset, called the \emph{vertex-inflation} of $S$.
\end{defin}

The complete analogues of Theorems~\ref{thmWedgeOfSpheres} and~\ref{thmWedgeDecomposition} are proved in~\cite{BjWW}.

\begin{lem}
The vertex inflation $(\Delta_V)_\na$ is homotopy equivalent to a wedge of $m(\nu,V)=\prod_{v\in V}(\nu(v)-1)$ many $(|V|-1)$-dimensional spheres.
\end{lem}

\begin{prop}\label{propWedgeForVertexInflation}
Let $S$ be a connected simplicial poset. Then, for any vertex-multiplicity function $\na\colon V(S)\to\Zo_+$, there is a homotopy equivalence
\[
|S_\na|\simeq |S|\vee \bigvee_{I\in S}(\Sigma^{\dim I+1}|\lk_SI|)^{\vee m(\na_I,V(I))}.
\]
\end{prop}

The constructions of vertex-inflation and edge-inflation can be naturally combined.

\begin{con}
Let $S$ be a simplicial poset, $\na\colon V(S)\to\Zo_+$ be a vertex multiplicity function, and $\ma\colon E(S)\to \Zo_+$ be an edge multiplicity function. First consider the edge-inflated poset $S_\ma$. Its vertex set naturally coincides with that of $S$: $V(S_\ma)=V(S)$. Therefore, we have a vertex multiplicity function $\na\colon V(S_\ma)=V(S)\to\Zo_+$ on the poset $S_\ma$. Hence we have a well defined simplicial poset $S_{\ma,\na}=(S_\ma)_\na$. Combining the definitions of a vertex-inflation and an edge-inflation we get the following characterization of simplices of $S_{\ma,\na}$. Elements of $S_{\ma,\na}$ have the form $(I,d_I,c_I)$ where $I\in S$, $d_I\colon V(I)\to \Zo_+$, $c_I\colon E(I)\to \Zo_+$ such that $d_I(v)\in[\na(v)]$ and $c_I(v)\in[\ma(v)]$. The partial order is  $(I,d_I,c_I)\leqslant (J,d_J,c_J)$ if $I\leqslant J$ and $(d_J)|_{V(I)}=d_I$ and $(c_J)|_{E(I)}=c_I$.
\end{con}

Next lemma gives a description of links in edge-inflated simplicial posets in terms of links of the original poset. Let $I\in S$ be a simplex and $\lk_SI$ be its link in the sense of Definition~\ref{definLinkInPoset}. In order to relate the elements (vertices and edges) of $\lk_SI$ to the corresponding elements of $S$ we introduce two maps: the map $\rest\colon V(\lk_SI)\to V(S)$ and the map $\rest\colon E(\lk_SI)\to E(S)$. Note that the vertex of $\lk_SI$ is a simplex $J\in S$ such that $I<J$ and $\dim J=\dim I+1$. There is a unique complementary vertex $v=J\setminus I\in V(S)$. We set $\rest(J)=J\setminus I$. Similarly we define $\rest\colon E(\lk_SI)\to E(S)$ by $\rest(J)=J\setminus I$ where $\dim J-\dim I=2$ in this case.

\begin{lem}\label{lemLinkOfEdgeInflated}
Let $I\in S$ be a simplex, $\ma\colon E(S)\to\Zo_+$ be an edge multiplicity function, and $c_I\colon E(I)\to \Zo_+$ be an edge counting function on $I$. Then
\[
\lk_{S_\ma}(I,c_I)\cong (\lk_SI)_{\tilde{\ma},\tilde{\na}}
\]
where $\tilde{\ma}\colon E(\lk_SI)\to\Zo_+$ is the composition of $\rest\colon E(\lk_SI)\to E(S)$ with $\ma\colon E(S)\to\Zo_+$, and $\tilde{\na}\colon V(\lk_SI)\to\Zo_+$ takes the values
\[
\tilde{\na}(J)=\prod_{e\in E(J)\setminus E(I)}\ma(e)
\]
for each $J>I$, $\dim J-\dim I=1$ (i.e. $J\in V(\lk_SI)$).
\end{lem}

This easily follows from counting all copies of simplices in $S_\ma$, containing the given simplex $(I,c_I)$. We now recall several notions from algebraic combinatorics.

\begin{defin}
Simplicial poset $S$ is called \emph{homotopically Cohen--Macaulay} (homotopically CM) if it is pure of dimension $n-1$, and the following conditions hold
\begin{enumerate}
  \item $S$ is homotopy equivalent to a wedge of $(n-1)$-dimensional spheres.
  \item For each simplex $I\in S$, the link $\lk_SI$ is homotopy equivalent to a wedge of $(n-2-\dim I)$-dimensional spheres.
\end{enumerate}
\end{defin}

\begin{defin}
Simplicial poset $S$ is called \emph{Cohen--Macaulay} (CM) over the base module $R$ if it is pure of dimension $n-1$, and the following conditions hold
\begin{enumerate}
  \item $\Hr_j(S;R)=0$ for $j<n-1$.
  \item $\Hr_j(\lk_SI;R)=0$ for each simplex $I\in S$ and $j<n-2-\dim I$.
\end{enumerate}
\end{defin}

We apply Lemma~\ref{lemLinkOfEdgeInflated} to prove the following Proposition, which is similar in spirit to the Cohen--Macaulay fiber theorem~\cite[Thm.5.1]{BjWW}.

\begin{prop}\label{propCMinherited}
Let $S$ be a CM simplicial poset and $\ma\colon E(S)\to\Zo_+$ be an edge multiplicity function. Then $S_\ma$ is also CM. The same statement holds for homotopically Cohen--Macaulay property.
\end{prop}

\begin{proof}
We prove the statement for homotopically Cohen--Macaulayness; the homological version is proved similarly. The proof is by induction on $\dim S$. The base $\dim S=0$ is trivially satisfied.

Since $S\simeq \bigvee S^{n-1}$ and $\lk_SI\simeq \bigvee S^{n-2-\dim I}$, the wedge decomposition of Theorem~\ref{thmWedgeDecomposition} states that $S_\ma$ is homotopy equivalent to a wedge of $(n-1)$-spheres.

Now we need to prove that any link in $S_\ma$ is also homotopy equivalent to a wedge of spheres. Let $(I,c_I)\in S_\ma$ be a simplex. By Lemma~\ref{lemLinkOfEdgeInflated} we have
\[
\lk_{S_\ma}(I,c_I)\cong (\lk_SI)_{\tilde{\ma},\tilde{\na}}=((\lk_SI)_{\tilde{\ma}})_{\tilde{\na}}.
\]
Since $S$ is homotopically CM, any link $\lk_SI$ is also homotopically CM by the definition. Therefore $(\lk_SI)_{\tilde{\ma}}$ is homotopically CM by induction hypothesis, in particular it is a wedge of spheres. Now, the wedge decomposition of Proposition~\ref{propWedgeForVertexInflation} tells that the additional vertex inflation $((\lk_SI)_{\tilde{\ma}})_{\tilde{\na}}$ is also homotopy equivalent to a wedge of spheres, which finishes the proof.
\end{proof}

\begin{rem}
We mention here that there is an operation of \emph{replicating} elements in a poset, which was introduced by Baclawski~\cite[Thm.7.3]{Bacl}. He proved that the CM property is preserved under replications. Although this operation seems similar to the operation of edge inflation in a simplicial poset, the latter cannot be reduced to replications in general.
\end{rem}

\section{Functoriality of wedge decompositions and persistent homology}\label{secFunctorialityAndPersistence}

In this section we restrict ourselves to flag complexes of mutigraphs and concentrate on functoriality issues needed in persistence theory. It is assumed that the vertex set $V$ of all multigraphs is fixed.

\begin{defin}
An (increasing) filtration of multigraphs is a sequence of inclusions
\begin{equation}\label{eqMGfiltr}
\tG_1\subset \tG_2\subset \cdots\subset \tG_s
\end{equation}
of multigraphs on the vertex set $V$, together with a nondecreasing sequence of nonnegative real numbers
\[
0\leqslant t_1\leqslant t_2\leqslant \cdots\leqslant t_s.
\]
\end{defin}

A filtration can be encoded by a single multigraph $\tG=\tG_s$ and a function $\tb\colon E(\tG)\to \Ro_+$, which indicates the birth time for each edge of $\tG$.

A filtration of multigraphs naturally induces the filtration of the corresponding flag complexes
\begin{equation}\label{eqFiltrationFlags}
|F.\tG_1|\subset |F.\tG_2|\subset \cdots\subset |F.\tG_s|.
\end{equation}
The latter is a filtration of topological spaces, hence one can define its persistent homology modules and persistence diagrams, see definition and details in~\cite{ZomCarl}.

In this section we develop a theory to compute persistent homology of filtration~\eqref{eqFiltrationFlags} using the wedge decomposition of Theorem~\ref{thmWedgeDecomposition}. The general idea is stated informally in the following construction.

\begin{con}\label{conIdeaOfPersistence}
By sorting birth-times of all edges, we can assume that, in the multigraph filtration~\eqref{eqMGfiltr} exactly one edge $e$ is added at each time moment. There is an alternative.

\begin{enumerate}
  \item The newborn edge $e$ connects two vertices which were previously disconnected. This means that a new edge appears in the underlying simple graph $G$, hence the topology of the underlying simplicial complex $K=F.G$ changes: some homology cycles may born, and some may die. By similar reasons, the topology of links $\lk_KI$ may change (this happens in the case when $e\cup I$ becomes a clique after adding $e$). Hence, some homology cycles may born or die in $\lk_KI$. Via the wedge decomposition~\eqref{eqWedgeInflated}, these changes induce the homotopy type change of $F.\tG$ (this poset is the edge inflation of $K$).
  \item The newborn edge $e$ connects two vertices $v_1,v_2$ which were already connected. In this case, the topologies of the underlying simplicial complex $K=F.G$ and its links do not change. However, the edge multiplicity function increases by 1 at a single position. This means, that the number $n(\ma,I)$ may increase by $\delta_I$ for some simplices $I\in K$ containing $\{v_1,v_2\}$ (we have $\delta_I\geqslant 0$ according to the monotonicity property, see Corollary~\ref{corMonot}). For all such $I$, each living homology of $\lk_KI$ generates $\delta_I$ many copies of itself.
\end{enumerate}
\end{con}

It can be seen that the computation of persistent homology of a filtration~\eqref{eqFiltrationFlags} reduces to the computation of persistent homology of the underlying filtration of simplicial complexes
\[
|F.G_1|\subset |F.G_2|\subset \cdots\subset |F.G_s|
\]
and the links of its simplices up to some arithmetical work with multiplicities. This can be easily parallelized. However, in order to justify the arguments in both items 1 and 2 above, one needs a functorial version of Theorem~\ref{thmWedgeDecomposition} and therefore a functorial version of the Poset fiber theorem~\ref{thmFiberThm}.

\begin{thm}[Functorial poset fiber theorem]\label{thmFuncFiberThm}
Consider a commutative diagram of posets
\[
\xymatrix{
P' \ar[d]^{f_1} \ar[r]^{g} & P''\ar[d]^{f_2}\\
Q' \ar@{^{(}->}[r]^{h}& Q''
}
\]
where $h$ is injective and both $f_1$ and $f_2$ satisfy the assumptions of the poset fiber theorem~\ref{thmFiberThm}:
\begin{itemize}
  \item for any $q'\in Q'$ the fiber $|f_1^{-1}(Q'_{\leqslant q'})|$ is $\dim |f_1^{-1}(Q'_{<q'})|$-connected; 
  \item for any $q''\in Q''$ the fiber $|f_2^{-1}(Q''_{\leqslant q''})|$ is $\dim |f_2^{-1}(Q''_{<q''})|$-connected;
  \item both $Q'$ and $Q''$ are connected.
\end{itemize}
Then the map $g\colon P'\to P''$ is homotopy equivalent to the map
\[
|Q'|\vee\bigvee_{q'\in Q'} |f_1^{-1}(Q'_{\leqslant q'})|\ast |Q'_{>q'}|\to |Q''|\vee\bigvee_{q''\in Q''} |f_2^{-1}(Q''_{\leqslant q''})|\ast |Q''_{>q''}|
\]
induced on the wedge summands by the map $h\colon |Q'|\to|Q''|$ and its restrictions: $|Q'_{\leqslant q'}|\to |Q''_{\leqslant h(q')}|$ and $|Q'_{>q'}|\to |Q''_{>h(q')}|$.
\end{thm}

The injectivity of $h$ is needed to ensure that restrictions of $h$ to $|Q'_{>q'}|\to Q''_{>h(q')}$ are well defined.

\begin{proof}
We simply note that the proof of the poset fiber theorem from~\cite{BjWW} is based on the properties of homotopy colimits hence it is functorial. The extended version of this argument is written below.

Let $D'\colon\cat(Q')\to\Top$ be the diagram of spaces, defined by $D'(q')=|f_1^{-1}(Q'_{\leqslant q'})|$ with morphisms induced by inclusions of subsets. Assume that a base point is chosen $x_{q'}\in D'(q')$ for each $q'\in Q'$. Let $\tD'\colon\cat(Q')\to\Top$ be the diagram with the same entries $\tD'(q')=D'(q')=|f_1^{-1}(Q'_{\leqslant q'})|$ but all nonidentity morphisms being the constant maps to the base points $x_{q'}$. The poset fiber theorem is proved by the sequence of homotopy equivalences
\[
|P'|= \colim D'\stackrel{\simeq}{\leftarrow} \hocolim D'\stackrel{\simeq}{\to} \hocolim \tD'\stackrel{\simeq}{\to} 
|Q'|\vee\bigvee_{q'\in Q'}(\tD'(q')\ast |Q'_{>q'}|)
\]
Here the first homotopy equivalence is due to $D'$ being cofibrant. The second equivalence follows because $D'\simeq \tD'$ (the asphericity of fibers is applied at this point). The last equivalence is the Wedge lemma~\cite[Proposition 3.5]{WZZ}.

A similar sequence of equivalences holds for the map $f_2\colon P''\to Q''$. Let $D'',\tD''\colon \cat(Q'')\to \Top$ be the diagrams with entries $D''(q'')=|f_2^{-1}(Q''_{\leqslant q''})|$, and inclusion morphisms in case of $D''$, and constant morphisms in case of $\tD''$. Then we have
\[
|P''|= \colim D''\stackrel{\simeq}{\leftarrow} \hocolim D''\stackrel{\simeq}{\to} \hocolim \tD''\stackrel{\simeq}{\to}
|Q''|\vee\bigvee_{q''\in Q''}(\tD''(q'')\ast |Q''_{>q''}|).
\]

Consider the additional diagrams $D^\diamond, \tD^\diamond\colon\cat(Q')\to\Top$ defined as compositions $D^\diamond = D''\circ h$ and $\tD^\diamond = \tD''\circ h$. The functoriality of the poset fiber theorem follows from commutativity of the diagram
\[
\xymatrix{
P' \ar[dd]^{g} \ar@{=}[r] & \colim D' \ar[d]  & \hocolim D' \ar[d] \ar[l]_{\simeq} \ar[r]^{\simeq} & \hocolim\tD' \ar[d] \ar[r]^(.3){\simeq} & |Q'| \vee \bigvee_{q'\in Q'}(|f_1^{-1}(Q'_{\leqslant q'})| \ast |Q'_{>q'}|)\ar[d] \\
& \colim D^\diamond \ar[d]  & \hocolim D^\diamond \ar[l]_{\simeq} \ar@{^{(}->}[d] \ar[r] & \hocolim \tD^\diamond \ar[d]\ar[r]^(.3){\simeq} & |Q'|\vee\bigvee_{q'\in Q'}(|f_2^{-1}(Q''_{\leqslant h(q')})|\ast |Q'_{>q'}|\ar[d]\\
P'' \ar@{=}[r] & \colim D''  & \hocolim D'' \ar[l]_{\simeq} \ar[r]^{\simeq} & \hocolim\tD'' \ar[r]^(.3){\simeq} & |Q''|\vee\bigvee_{q''\in Q''}(|f_2^{-1}(Q''_{\leqslant q''})|\ast |Q''_{>q''}|)
}
\]
Each square of this diagram is commutative either by definition or by the properties of colimits and homotopy colimits.
\end{proof}

\begin{thm}[Functorial wedge decomposition]\label{thmFuncWedgeOfSpheres}
Let $h\colon S'\to S''$ be an inclusion of simplicial posets, and let $\ma'\colon E(S')\to\Zo_+$ and $\ma''\colon E(S'')\to\Zo_+$ be edge-multiplicity functions satisfying inequality $\ma'(I')\leqslant \ma''(h(I'))$ for each edge $I'\in E(S')$, and hence inducing the inclusion 
\[
g\colon S'_{\ma'}\to S''_{\ma''}
\]
of edge-inflated simplicial posets. Assume for simplicity that both $|S'|$ and $|S''|$ are connected. Then $g$ is homotopy equivalent to the map
\[
|S'|\vee \bigvee_{I'\in S', \dim I'\geqslant 1}(\Sigma^{\dim I'+1}|\lk_{S'}I'|)^{\vee n(\ma_{I'},I')}\to |S''|\vee \bigvee_{I''\in S'', \dim I''\geqslant 1}(\Sigma^{\dim I''+1}|\lk_{S''}I''|)^{\vee n(\ma_{I''},I'')}.
\]
induced by $h\colon S'\to S''$ and its restrictions to links.
\end{thm}

This statement follows directly from Theorem~\ref{thmFuncFiberThm} and Proposition~\ref{propMonot}, which asserts that the maps of spaces $|S'_{\leqslant I'}|\to|S''_{\leqslant I''}|$ induced by $g$ are homotopic to inclusions of subwedges of spheres.

Theorem~\ref{thmFuncWedgeOfSpheres} justifies the idea of persistent homology calculation described in Construction~\ref{conIdeaOfPersistence}. The precise algorithm to compute persistent homology of a multigraph and its realization will be described in a different paper.

\section{General inflations}\label{secGeneralInflations}

The only reason why we restricted our attention to edge-inflations so far, is their relation to clique complexes of multigraphs, hence possible applications in topological data analysis. However, from the mathematical point of view it is natural to consider inflations of simplices of any dimension. In this section we introduce a general inflation process on a simplicial poset.

Given a set $V$, let us call \textit{inflation function} on $V$ a function which maps finite subsets of $V$ to non-negative integers with the only condition that the empty set is mapped to one.

\begin{con} (Inflated complex $M_\mu$.)

Let $\mu$ be an inflation function on $V$.
We will denote its value on subset $U = \{u_1,u_2,\dots, u_k\} \subset V$ as $\mu_{u_1,u_2\dots,u_k} = \mu_U$.
We begin with the set of vertices of $M_\mu$.
For any $v\in V$ consider $\mu_v$ vertices labeled $v$.
Then each pair of these vertices is connected with $\mu_{v_1, v_2}$ edges if their labels are $v_1$ and $v_2$.
Then each triple of these edges forming (the boundary of) a triangle is glued with $\mu_{v_1, v_2, v_3}$ $2$-simplices if the labels of their vertices are $v_1, v_2, v_3$. We proceed further in a similar way. At $n$-th step each $n+1$ of the $(n-1)$-simplices forming the boundary of $n$-simplex is glued with $\mu_{v_1, \dots, v_{n+1}}$ $n$-simplices if the labels of their vertices are $v_1, \dots, v_{n+1}$.

The resulting space can be considered as a simplicial poset and will be denoted by $M_\mu$.
\end{con}

\begin{rem}
If $\mu_U=0$ for some $U = \{u_1,u_2,\dots,u_k\} \subset V$ then the simplicial poset $M_\mu$ does not contain any simplex on vertices with labels $u_1,\dots,u_k$.
Therefore for $U'\supset U$ there will be no corresponding simplices and without loss of generality we may consider only the inflation functions with zero value on such $U'$.
\end{rem}

The following statement seems to be obvious.

\begin{lem}\label{lemInflSimplexStructure}
The simplices of $M_\mu$ are pairs $(I, c)$ where $I$ is a finite subset of $V$ and $c: I'\mapsto c_{I'}$ is a function from nonempty subsets of $I$ to positive integers such that for all $I'$ there holds $c_{I'}\le \mu_{I'}$.
The simplex $(I, c)$ is a subsimplex of $(I', c')$ if and only if $I$ is a subset of $I'$ and $c$ in a restriction of $c'$.
\end{lem}

For an inflation function $\mu$ and subset $I\subset V$ let us introduce the \textit{induced inflation function} $\mu_{/I}$ to be an inflation function for the set $V\setminus I$ mapping a non-empty $J\subset V\setminus I$ to $(\mu_{/I})_J = \prod_{I'\le I}\mu_{I'\sqcup J}$, the value on the empty set is set to be unit.

\begin{lem}\label{lemInflatedLink}
The link of the simplex $(I, c)$ in $M_\mu$ is isomorphic to an inflated complex $M_{\mu_{/I}}$ on the vertex set $V\setminus I$ and inflation function $\mu_{/I}$.
\end{lem}

\begin{proof}
Follows directly from lemma \ref{lemInflSimplexStructure}.
\end{proof}

\begin{lem}\label{lemAddInflation}
Let $V$ be a vertex set and $\mu, \nu$ be two inflation functions such that there exist exactly one $I\subset V$ such that $\mu_I\neq \nu_i$ and for this $I$ there holds $\nu_I=\mu_I+1>1$.
Suppose also that $M_\mu$ is connected.
Then $$M_\nu \simeq M_\mu\vee (\Sigma^{\dim I + 1} \operatorname{lk}_{M_\mu} I)^{\vee \prod_{I'< I} \mu_{I'}}.$$
The inclusion $M_\mu\subset M_\nu$ is consistent with the inclusion of $M_\mu$ onto the first wedge summand.
\end{lem}

\begin{proof}
If we delete from $M_\nu$ all simplices $(I, c)$ for all admissible $c$ such that $c_I = \nu_I$ together with all simplices containing them, we obtain $M_\mu$.
The number of such simplices is exactly $\prod_{I'< I} \mu_{I'}$.
Let us prove that adding any such a simplex $(I, c)$ together with its star is equivalent to wedge addition of $\Sigma^{\dim I + 1} \operatorname{lk}_{M_\mu} I$.

Indeed, $\operatorname{st}_{M_\nu} (I, c)$ is isomorphic to $\Delta^{\operatorname{dim}I} * \operatorname{lk}_{M_\nu} (I, c)/\,\,\sim$ where the equivalence relation comes from the gluing map $f:(\partial\Delta^{\operatorname{dim}I}) * \operatorname{lk}_{M_\nu} (I, c)\to M_\mu$.
But this gluing map is absolutely the same for the simplex $(I, c')$ where $c'$ coincides with $c$ except at $c'_I = 1$.
Therefore $f$ is null-homotopic and $M_\mu\cup_f (\Delta^{\operatorname{dim}I} * \operatorname{lk}_{M_\nu} (I, c)) \simeq M_\mu \vee (\Delta^{\operatorname{dim}I} * \operatorname{lk}_{M_\nu} (I, c))/(\partial(\Delta^{\operatorname{dim}I} * \operatorname{lk}_{M_\nu} (I, c)) \simeq M_\mu \vee S^{\operatorname{dim} I} * \operatorname{lk}_{M_\nu} (I, c)$, which concludes the proof.
\end{proof}

The following lemma is a generalization of Theorem \ref{thmWedgeDecomposition}.

\begin{lem}\label{lemInflatedWedgeDecomp}
Let $V$ be a vertex set and $\mu$ be an inflation function.
Denote by $\mu^\circ$ the inflation function on $V$ mapping $I$ to $\sigma(\mu_I)$ where $\sigma(x)$ is zero when $x=0$ and one when $x>0$.
Therefore $\mu^{\circ}$ is a characteristic function of simplices of the ordinary simplicial complex $M_{\mu^\circ}$.
Then 
\begin{enumerate}
\item 
    \begin{align}\label{ma}
        M_\mu \simeq \bigsqcup_{W\in \operatorname{comp}(\mu)}\big( M_{\mu^\circ|_W} \vee \bigvee_{I\subseteq W, I\neq \emptyset}(\Sigma^{\dim I + 1} \operatorname{lk}_{M_{\mu^\circ}} I)^{\vee N(\mu, I)} \big)
    \end{align}
    where $N(\mu, I) = \sum_{I'\le I} (-1)^{\dim I - \dim I'} \prod_{I''\le I'}\mu_{I''}$,
    the disjoint union is over the vertex sets of connected components of $M_{\mu^\circ}$ and $\mu^\circ|_W$ means the restriction of $\mu^{\circ}$ to the vertex set $W$.
    
\item if $M_\mu$ is connected (which is a simple condition on $\mu^{\circ}$) we have 
    \begin{align}\label{ma_connected}
        M_\mu \simeq \bigvee_{I\subseteq V}(\Sigma^{\dim I + 1} \operatorname{lk}_{M_{\mu^\circ}} I)^{\vee N(\mu, I)}
    \end{align}

\item for disconnected $M_\mu$ holds the following refinement of the previous formula:
    \begin{align}\label{ma_suspension}
        \Sigma M_\mu \simeq \bigvee_{I\subseteq V}(\Sigma^{\dim I + 2} \operatorname{lk}_{M_{\mu^\circ}} I)^{\vee N(\mu, I)}
    \end{align}
\end{enumerate}
\end{lem}

\begin{rem}
In this formula we consider links in $M_{\mu^{\circ}}$ of some subsets $I$ which possibly do not form a simplex, but actually $N(\mu, I)$ for such subsets is zero, so these undefined links are ignored.
Equivalently the wedge could be considered over only the simplices of the simplicial complex $M_{\mu^{\circ}}$.
\end{rem}

\begin{proof}

Formulas (\ref{ma}) and (\ref{ma_connected}) are equivalent for the case of connected $M_\mu$ (because $M_{\mu^{\circ}} = \operatorname{lk}_{M_{\mu^{\circ}}} I$).
The formula (\ref{ma_suspension}) is also a consequence of the formula (\ref{ma}).
Then it is sufficient to prove the first part.

The proof goes by induction on the number of vertices and on the number of `extra inflations', i.e. $\sum_I(\max(\mu_I-1, 0))$.
For the induction on the number of vertices the base is obvious.
For the induction on the number of extra inflations the base is when there are no extra inflations, then $\mu = \mu^{\circ}$ and $N(\mu, I)=0$ for all nonempty $I$ which make the formula obvious.

It is sufficient to prove the inductive step for adding one extra inflation as in lemma~\ref{lemAddInflation}.
Restricting our attention to the connected component of this extra inflation, we may pass to the case when $M_\mu$ is connected.
Thus by lemma~\ref{lemAddInflation}, $M_\nu \simeq M_\mu\vee (\Sigma^{\dim I + 1} \operatorname{lk}_{M_\mu} I)^{\vee \prod_{I'< I} \mu_{I'}}.$
From lemma~\ref{lemInflatedLink} we have $\operatorname{lk}_{M_\mu} I) \simeq M_{\mu_{/I}}$.
And by the induction hypothesis (in form of formula (\ref{ma_suspension}))
\begin{align*}
    \Sigma M_{\mu_{/I}}
    \simeq 
    \bigvee_{J\subseteq V\setminus I}(\Sigma^{\dim J +2} \operatorname{lk}_{M_{\mu_{/I} {} ^\circ}} J)^{\vee N(\mu_{/I}, J)}
    % \bigsqcup_{W\in\operatorname{comp}(A_{/I})}(M_{A_{/I}^\circ |_W} \vee
    % \bigvee_{J\subset W, J\neq \emptyset} (\Sigma^{\# I} \operatorname{lk}_I(M_{A_{/I}^\circ})^{\vee N(A_{/I}, J)})
    %\bigvee_{J\subset V\setminus I}(\Sigma^{\# J} \operatorname{lk}_J(M_{A_{/I} {} ^\circ}))^{\vee N(A_{/I}, J)}
\end{align*}
Clearly, 
$\operatorname{lk}_{M_{\mu_{/I} {} ^\circ}} I$ is $\operatorname{lk}_{M_{\mu^\circ}} (I\sqcup J)$.

Then by lemma~\ref{lemAddInflation} we have
\begin{align*}
    M_\nu = M_\mu\vee \bigvee_{J\subseteq V\setminus I} (\Sigma^{\dim I + \dim J + 2} \operatorname{lk}_{M_{\mu^\circ}} (I\sqcup J))^{\vee (N(\mu, J) \prod_{I'< I} \mu_{I'})}
\end{align*}

By the induction hypothesis we have $M_\mu \simeq \bigvee_{I\subset V}(\Sigma^{\dim I + 1} \operatorname{lk}_{M_{\mu^\circ}} I)^{\vee N(\mu, I)}$.
Then our lemma follows from the fact that for any $J\subset V\setminus I$ there holds 
\[
N(\nu, I\sqcup J) - N(\mu, I\sqcup J) = N(\mu_{/I}, J) \prod_{I'<I} \mu(I'),
\] 
which can be proved directly.
\end{proof}

Actually this proof can be used in a more general setting and give a functorial version of lemma~\ref{lemInflatedWedgeDecomp} which is a generalisation of Theorem~\ref{thmFuncWedgeOfSpheres} for (not necessarily connected) inflated complexes.

\begin{rem}
Absolutely similarly to Proposition~\ref{propCMinherited} it can be proved that the properties of Cohen--Macaulayness and homotopically Cohen--Macaulayness are inherited by general inflations of simplicial complexes. 
\end{rem}

Let $\mathcal{IF}$ denote \textit{the category of inflation functions} where objects are pairs $(V, \mu)$ of a set and an inflation function on it, and morphisms from $(V, \mu)$ to $(W, \nu)$ are injective maps $f:V\to W$ such that for any $v_1, \dots, v_n\in V$ there holds $\mu_{v_1, \dots, v_n}\le \nu_{f(v_1),\dots,f(v_n)}$.
Morphisms in the category $\mathcal{IF}$ will be called \textit{non-decreasing maps of inflation functions}.

Let us introduce also the category $\mathcal{WH}$ of disjoint unions of wedges of topological spaces up to homotopy equivalence.
Namely, an object of $\mathcal{WH}$ is the homotopy class of a topological space $X$ together with its homotopy decomposition $X\simeq \bigsqcup_{\alpha\in a} \bigvee_{\beta\in b(\alpha)} X_\beta$.
The morphisms in $\mathcal{WH}$ are homotopy classes of continuous maps $f: \bigsqcup_{\alpha\in a} \bigvee_{\beta\in b(\alpha)} X_\beta \to \bigsqcup_{\alpha'\in a'} \bigvee_{\beta'\in b'(\alpha')} X'_{\beta'}$ such that for each $\alpha\in a, \beta\in b(\alpha)$ there exist $\alpha'\in a', \beta'\in b'(\alpha')$ and $f_\beta: X_\beta\to X'_{\beta'}$ for which the map $f|_{X_\beta}$ is a composition of $f_\beta$ and the inclusion $X'_{\beta'}\subset \bigsqcup_{\alpha'\in a'} \bigvee_{\beta'\in b'(\alpha')} X'_{\beta'}$.

There exists an obvious forgetful functor $\mathcal{WH}\to \mathcal{H}$ to the homotopy category.

\begin{thm}\label{thmFunctorialInflation}
The formula (\ref{ma}) yields a map 
$$\mu\mapsto M_\mu\simeq \bigsqcup_{W\in \operatorname{comp}(\mu)}\big( M_{\mu^\circ|_W} \vee \bigvee_{I\subseteq W, I\neq \emptyset}(\Sigma^{\dim I + 1} \operatorname{lk}_{M_{\mu^\circ}} I)^{\vee N(\mu, I)} \big)$$
in which the right hand side can be considered as a decomposition in the disjoint union of wedges, therefore, as an object of $\mathcal{WH}$.
This construction yields actually a functor $q: \mathcal{IF}\to \mathcal{WH}$.
The composition of this functor with the forgetful functor $\mathcal{WH}\to \mathcal{H}$ is equivalent to the geometric realization functor.
\end{thm}

\begin{proof}
The interesting part of this theorem is to define 1) the functor $q$ from $\mathcal{IF}$ to $\mathcal{WH}$ 2) the functor $M$ from $\mathcal{IF}$ to the category $SimPoset$ of simplicial posets.

In order to give a precise definition of a functor $q$ we have to show which wedge summands map to which and how.
For simplicity we will do that only for morphisms $i:\mu\to \nu$ of inflation functions on the common vertex set $V$ with identity function from $V$ to itself.
The fact that our definition is compatible to renaming of vertices is quite simple.
Also we restrict our attention to the case of connected $M_\mu, M_\nu$, because the general case easily reduces to it.

For simplices $I$ of $M_{\mu^{\circ}}$ there could be many (precisely, $N(\mu, I)$) wedge summands $\Sigma^{\dim I + 1} \lk_{M_{\mu^\circ}} I$ in $M_A$, so we have to label them.
Consider the set of lists $(a_{I'})_{I'\le I}$ where all $1\le a_{I'}\le \mu_{I'}$.
Such a list will be called \textit{a reducible $(\mu, I)$-list} if there exists $v\in I$ for which $a_{I'}=1$ for all $I'$ containing $v$.
Otherwise the list is called \textit{an irreducible $(\mu, I)$-list}.
A direct calculation shows that the number of irreducible $(\mu, I)$-lists is equal to $N(\mu, I)$, so we label our wedge summands corresponding to simplex $I$ with irreducible $(\mu, I)$-lists.

Then we construct a wedge morphism 
$$q(i): \bigvee_{I\subseteq V}(\Sigma^{\dim I + 1} \operatorname{lk}_{M_{\mu^\circ}} I)^{\vee N(\mu, I)}\to \bigvee_{I\subseteq V}(\Sigma^{\dim I + 1} \operatorname{lk}_{M_{\nu^\circ}} I)^{\vee N(\nu, I)}.$$
The wedge summand $\Sigma^{\dim I + 1} \operatorname{lk}_{M_{\mu^\circ}} I$ with label $(a_{I'})_{I'\le I}$ goes by an obvious inclusion to $\Sigma^{\dim I + 1} \operatorname{lk}_{M_{\nu^\circ}} I$ with the same label.

Then we have to precisely define the geometric realization functor from $\mathcal{IF}$ to category $SimPoset$ of simplicial posets which maps $\mu\mapsto M_\mu$.
For a morphism $f$ from $(V, \mu)$ to $(W, \nu)$ in $\mathcal{IF}$ we map a simplex $(I, c)$ (here $I\subset V$, $c:2^I\to \mathbb{N}$) linearly to simplex $(f(I), c\circ f)$.

Now it remains to prove that these functors $q$ and $M$ are homotopically equal, i.e. that the following diagram with obvious functors to the homotopy category commutes. 
\[
\xymatrix{
\mathcal{IF} \ar[d]^{M} \ar[r]^{q} & \mathcal{WH} \ar[d]^{}\\
SimPoset \ar[r]^{} & \mathcal{H}
}
\]
The commutativity at the level of objects is just lemma~\ref{lemInflatedWedgeDecomp}, let us prove the commutativity at the level of morphisms.
Any morphism in $\mathcal{IF}$ decomposes to morphisms adding one to some $\mu_I$, therefore we can consider only morphisms from $(V, \mu)$ to $(V, \nu)$ with $\nu$ different from $\mu$ exactly at one simplex $I$ and the difference is unit.
The case when for some $I$ we have $\mu_I = 0, \nu_I = 1$ follows from general arguments.
The case when for some $I$ we have $\nu_I=\mu_I+1>1$ follows from the analysis of the proof of lemma~\ref{lemAddInflation} and lemma~\ref{lemInflatedWedgeDecomp}.
Indeed, we glue to $M_\mu$ the stars of simplices $(I, c)$ with $c_I=\nu_I$ and this gluing is equivalent to wedge addition of something to $M_\mu$.
So it does not affect the decomposition of $M_\mu$ itself to a wedge sum.

\end{proof}

\begin{rem}
There is another (probably, more interesting) category related to inflated complexes.
One can consider the category $\mathcal{IF}_g$ of geometric realizations of inflation functions, where morphisms are embeddings of simplicial posets resembling non-decreasing maps of inflation functions.
This category is more rich: while $\mathcal{IF}$ has only one isomorphism for an object $(V, \mu)$ projecting to identity map $V\to V$, that is the identity isomorphism, the category $\mathcal{IF}_g$ has plenty of such isomorphisms, namely $\prod_I (\mu_I !)^{\prod_{I'<I}\mu_I'}$.
A natural question is whether our wedge decomposition is functorial for this category, and the answer is, unfortunately, negative.
Precisely, there is no functor $\mathcal{IF}_g\to\mathcal{WH}$ making the following diagram commutative.
% compatible with the functor $q$ and the obvious inclusion $\mathcal{IF}\to \mathcal{IF}_g$.
\[
\xymatrix{
\mathcal{IF} \ar[d]^{} \ar[r]^{q} & \mathcal{WH} \ar[d]^{}\\
\mathcal{IF}_g \ar[r]^{}\ar@{.>}[ur] |{+} & \mathcal{H}
}
\]
\end{rem}

\end{document}